\documentclass{amsart}
\title{Thermoacoustic Tomography in Elastic Media}
\author{Justin Tittelfitz}
\address{University of Washington}
\email{jtittelf@math.washington.edu}
\date{April 21, 2011}

\usepackage{epsfig}
\usepackage{graphicx}
\usepackage{amssymb}
\usepackage{verbatim}

\newtheorem{theorem}{Theorem}[section] 
\newtheorem{lemma}[theorem]{Lemma} 

\theoremstyle{definition} 
\newtheorem{definition}[theorem]{Definition}

\newcommand{\R}{\mathbb{R}}
\newcommand{\Rn}{\R^{n}}

\newcommand{\norm}[1]{\left\Vert#1\right\Vert}
\newcommand{\abs}[1]{\left\vert#1\right\vert}
\newcommand{\set}[1]{\left\{#1\right\}}

\newcommand{\grad}{\nabla}

\newcommand{\divergence}{\nabla \cdot}

\newcommand{\innp}[2]{\langle #1,#2\rangle}

\newcommand{\tr}{\mathrm{tr}}

\newcommand{\wt}[1]{\widetilde{#1}}

\newcommand{\HDO}{H_D(\Omega)}
\newcommand{\LTO}{L^2(\Omega)}

\begin{document}

\maketitle

\begin{abstract}
We investigate the problem of recovering the initial displacement~$f$ for a solution $u$ of a linear, isotropic, non-homogeneous elastic wave equation, given measurements of $u$ on $[0,T]\times \partial \Omega$, where $\Omega\subset\R^3$ is some bounded domain containing the support of  $f$.  For the acoustic wave equation, this problem is known as thermoacoustic tomography (TAT), and has been well-studied; for the elastic wave equation, the situation is somewhat more subtle, and we give sufficient conditions on the Lam\'e parameters to ensure that recovery is possible.
\end{abstract}

\section{Introduction}
In this paper, we consider the linear, isotropic elastic wave equation and Cauchy initial value problem given by
\begin{align}
	\begin{cases}
		(\partial_t^2 + P) u &= 0 \quad \textrm{  in  } (0,T) \times \R^3, \\
		u|_{t=0} &= f, \\
		\partial_t u |_{t=0} &= 0,
	\end{cases}
\label{PDEu}
\end{align}
where $u = (u_1,u_2,u_3)$ is the displacement vector,
\begin{align*}
	-Pu = \divergence \left( \mu(x) ((\grad u) + (\grad u)^T)\right) + \grad (\lambda(x) \divergence u),
\end{align*}
$\lambda$ and $\mu$ are the Lam\'e parameters, 
\begin{align*}
	\left(\grad u\right)_{i,j} = \frac{\partial u_i}{\partial x_j}
\end{align*}
is the Jacobian of $u$ and $(\grad u)^T$ is its transpose.  Equivalently, we can write
\begin{align*}
	-Pu = \mu \left( \Delta u + \nabla (\divergence u) \right) + \lambda \grad (\divergence u) + (\divergence u) \grad \lambda + \sum_{j=1}^3 \nabla \mu \cdot \left( \nabla u_j + \partial_j u \right) e_j.
\end{align*}
We will assume that $\lambda(x)$ and $\mu(x)$ are independent of $t$; we will also assume that they are each positive, in that there is some positive constant $\alpha_0$ so that $\lambda,\mu > \alpha_0$, and later, we will make further assumptions about bounds and smoothness.
Throughout, we will assume that our initial-data, $f$, is compactly supported in a set $\Omega$, which could, in general, be any bounded domain in $\R^3$, though for simplicity, we will take $\Omega = B(0,R)$ for some $R>0$.  The objective of thermoacoustic tomography in elastic media is to recover $f$, given the data
\begin{align*}
	\Lambda f := u|_{[0,T]\times \partial \Omega}.
\end{align*}
In particular, we will show that, if the Lam\'e parameters satisfy a condition on their relative size (briefly, that $\sup \sqrt{\mu} < 3 \inf\sqrt{\lambda + 2\mu}$) and one regarding their gradients (to be discussed later), then $f$ is recoverable via a Neumann series. Our method is strongly inspired by the techniques of thermoacoustic tomography for scalar wave equations, and we adapt the time-reversal approach used by Stefanov and Uhlmann in \cite{SU}. We devote the next section to discussing the existing results and methods in thermoacoustic tomography, and how they motivate the work in this paper.
\section{Background and Thermoacoustic Tomography}
Thermoacoustic tomography (hereafter TAT) is a method of medical imaging where the object of interest is exposed to a short electromagnetic (EM) pulse, absorbing some of the EM energy in the process.  
Because cancerous cells will absorb more of this energy than healthy cells will, it would be diagnostically useful to know the absorption function $a(x)$. To accomplish this, TAT makes use of the elastic expansion in the nearby tissue caused by the energy absorption, as well as the fact that this initial pressure distribution is roughly proportional to the absorption distribution. This initial pressure, in turn, leads to a pressure wave $p(t,x)$ that propagates through the object, and is then measured by transducers located on an observation surface $\Gamma$ surrounding the object for some length of time, the goal being to use this data to reconstruct the initial pressure. In much of the literature, the intended application is the imaging of cancer in a human breast, though some authors have written specifically about imaging the brain. This is somewhat more complicated, as the skull introduces a jump discontinuity in the sound speed, but under certain assumptions, recovery is still quite possible (see \cite{SU2}, \cite{QSUZ}).

Mathematically, we consider the scalar wave equation and Cauchy initial value problem
\begin{align*}
	\begin{cases}
		(\partial_t^2 + A) p &= 0 \quad \textrm{  in  } (0,T) \times \R^n, \\
		p|_{t=0} &= f, \\
		\partial_t p |_{t=0} &= 0,
	\end{cases}
\end{align*}
where $A(x,D) = - c^2(x) \Delta$,  (or more generally, $A(x,D) = -c^2(x)\Delta_g$, for some other metric $g$ on $\Rn$). The initial pressure function $f$ is typically assumed to be compactly supported inside of some bounded domain $\Omega \subset \Rn$ (corresponding to the object to be imaged), though in some work, $f$ is taken to be supported in \textit{some} compact set, or even merely $f \in L^p$ for $p > 2n/(n-1)$ (see \cite{ABK}). The observation surface $\Gamma$ is often taken to be $\partial \Omega$ (sometimes referred to as \textit{complete data}), though there are also satisfactory results for the case where $\Gamma$ is some other set, such as a portion of $\partial \Omega$ (likewise, \textit{incomplete data}; see, for instance, \cite{SU}, \cite{XKA}, \cite{XWKA}). The data one then collects is 
\begin{align*}
	\Lambda f := \set{ p(t,y): 0\leq t \leq T, y \in \Gamma}
\end{align*}
and from this, the goal of TAT is to recover $f$.

There are three main recovery methods used in TAT, and their applicability depends largely on the assumptions made about the geometry and physical attributes of the medium (in terms of $\Omega$ and $c(x)$) and the observation surface $\Gamma$. We will briefly discuss some of these here; for a detailed account and comparison of these techniques, including their relative advantages and disadvantages, see the excellent survey papers by Hristova, Kuchment and Nguyen \cite{HKN}, or Kuchment and Kunyansky \cite{KK1}, \cite{KK2}.

In the method of filtered backprojection, one assumes that $c$ is constant (i.e. the medium is acoustically homogeneous) and that the observation surface $\Gamma$ is a sphere of radius $R$.  We can then recover $f$ through integral formulas such as
\begin{align*}
	f(x) = - \frac{1}{8\pi^2 R}\Delta \int_\Gamma \frac{h(\abs{y-x},y)}{\abs{y-x}} dA(y),
\end{align*}
where 
\begin{align*}
	h(r,y) = \int_{\mathbb{S}^{n-1}} f(y + r\omega) r^{n-1} \ d\omega, \quad y\in \Gamma
\end{align*}
are the spherical integrals of $f$, with $dA$ and $d\omega$ are the surface measures on the respective spheres. For more on the integral geometry approach to TAT, see the work of Agranovsky, Berenstein and Kuchment \cite{ABK}, Finch, Patch and Rakesh \cite{FPR}, Xu and Wang \cite{XW}, Finch, Haltmeier and Rakesh \cite{FHR} and Kunyansky \cite{K2}.

The method of eigenfunction expansion applies to a slightly more general setting, in theory allowing $c$ to be variable, and for $\Gamma$ to be any closed surface (i.e. $\Gamma = \partial \Omega$ for some $\Omega$). We then seek to write $f$ as a Fourier series
\begin{align*}
	f(x) = \sum f_k \psi_k(x)
\end{align*}
where $\psi_k$ are eigenfunctions of the operator $-c^2(x) \Delta$ in $\Omega$ with Dirichlet boundary conditions on $\partial \Omega$, and then find the coefficients $f_k$ using integral formulas. In \cite{K}, Kunyansky showed, in the case $\Gamma$ is a cube and $c$ is constant, that $f$ can be recovered fast and precisely. Of course, for a complicated set $\Omega$, or a variable speed $c(x)$, the eigenfunctions and eigenvalues may not be known, and it is not clear whether this method can be effectively implemented.

The third method (and the method we will eventually use) is known as time-reversal, and was first proposed by Finch, Patch and Rakesh in \cite{FPR}, and first implemented by Burgholzer, Matt, Haltmeier and Paltauf in \cite{BMHP}. Here, $c$ is allowed to be variable (i.e. the medium can be non-homogeneous), and the restrictions on the geometry of $\Gamma$ are far less than those of the other two methods. The key assumption is that there is good local energy decay, meaning that for $f$ compactly supported in $\Omega$, the energy of the solution in $\Omega$ decays sufficiently fast as $t$ increases. To better illustrate the nature of this requirement, we first discuss an ideal case: the constant speed wave equation $(\partial^2_t - \Delta) p = 0$ in $\R^+ \times \R^3$. In this setting, Huygen's principle would apply, and we would know that for some $\wt T$, $u(t,x) = 0$ inside $\Omega$ for $t > \wt T$. In this case, we could recover $f$ by considering solutions of the initial value problem
\begin{align}
	\begin{cases}
		(\partial_t^2 - \Delta) q &= 0 \quad \textrm{  in  } \R^+ \times \R^n, \\
		q|_{t=\wt T} = q_t|_{t = \wt T} &= 0, \\
		q|_{\partial \Omega} = g(t,y)
	\end{cases}
\label{HUY}
\end{align}
where $g = \Lambda f$, and then solving the problem in the reverse time direction, since $q(0,x) = f(x)$ by uniqueness.

Of course, Huygen's principle does not apply in the general setting, but there is a useful analog which still leads to strong results. To ensure the kind of energy decay we need, we will assume that the speed is \textit{non-trapping}, meaning that all rays starting in $\Omega$ leave in finite time, and that the supremum of these times is also finite. Explicitly, for the Hamiltonian $H(x,\xi) = \frac{1}{2}c^2(x) \abs{\xi}^2$, we consider the solutions in $\R^{2n}_{x,\xi}$ of the system
\begin{align*}
	\begin{cases}
		x'_t = \frac{\partial H}{\partial \xi} = c^2(x)\xi, \\
		\xi'_t = -\frac{\partial H}{\partial x} = - \frac{1}{2} \grad (c^2(x)) \abs{\xi}^2, \\
		x|_{t=0} = x_0, \quad \xi|_{t=0} = \xi_0
	\end{cases}
\end{align*}
with $x_0 \in \Omega$ and $\xi_0 \neq 0$ (these solutions are called bicharacteristics, and their projections to $\Rn_x$ are often called rays). We say that $(\Omega, c)$ is non-trapping if each ray leaves $\Omega$ in finite time, and the supremum of these times is finite as well. For a non-trapping $(\Omega, c)$, we will call the supremum $T(\Omega)$. 

One may gain additional understanding of the nature of this definition by considering an example of a trapping metric: consider $c(x) = \abs{x}$, with $\Omega$ an annulus $\set{x: r_1 < \abs{x} < r_2}$. Then, for $x_0 \in \Omega$, if $\xi_0$ is perpendicular to $x_0$, it is straightforward to check that the resulting ray is $\set{x: \abs{x} = \abs{x_0}}$ (see Figure \ref{fig2}), which will remain inside $\Omega$ for all time.

\begin{figure}
\begin{center}
\includegraphics[width=5cm]{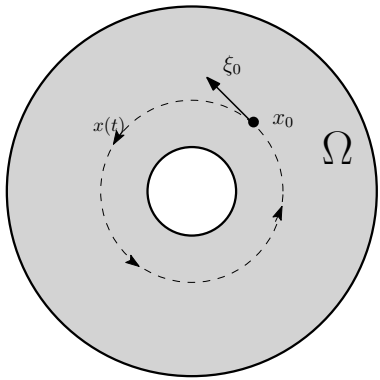}
\end{center}
\caption{An example of a trapped geodesic.}
\label{fig2}
\end{figure}

Because singularities propagate along bicharacteristics, the non-trapping hypothesis also ensures that any singularities of $f$ will have left $\Omega$ by $T(\Omega)$, or put another way, that $p(t,x)$ is a smooth function on $\Omega$ for $t > T(\Omega)$. This has been discussed extensively in the literature; see for instance \cite{D}, \cite{S} or \cite{T}.\label{microlocal}

Because we can no longer assume that $p$ eventually vanishes inside $\Omega$, it would be incorrect to think that, for any finite time $\wt T$, solutions $q$ of \eqref{HUY} will satisfy $q(0,x) = f(x)$, at least not exactly. To alleviate this, some authors have replaced the boundary condition by $q|_{\partial \Omega} = \chi(t) g(t,y)$, where $\chi$ is a smooth cutoff function vanishing near $\wt T > T(\Omega)$, and $\chi = 1$ near $(-\infty, T(\Omega))$ (see \cite{H}). In \cite{SU}, Stefanov and Uhlmann used a different time-reversal method (this will be discussed thoroughly in the final section of this paper), and showed that, for the acoustic wave equation with variable sound speed, it is possible to invert $\Lambda$ using a Neumann series. Subsequently, in \cite{QSUZ}, Qian, Stefanov, Uhlmann and Zhou went on to give a specific numerical algorithm for recovery based on this theoretical understanding. In this paper, we will follow the method of \cite{SU} when possible, assuming that the absorption of EM energy and subsequent elastic tissue expansion lead to elastic wave propagation instead of acoustic. 

The differences between scalar equations and systems will introduce some difficulties, the main difficulty being the unique continuation problem; i.e., determining when specifying Cauchy data on a hypersurface $S \subset \R\times\R^{3}$ is sufficient to uniquely determine the solution in a neighborhood of $S$. For the classical case (where the coefficients are smooth), this is Holmgren's theorem (see, for instance, \cite{TAY}). In the case of non-smooth coefficients for a scalar wave equation, there are good results due to Tataru (see \cite{T1},\cite{T2}). For the static Lam\'e system (i.e. the elliptic system $Pu = 0$), Lin, Nakamura, Uhlmann and Wang have recently shown (in \cite{LNUW}) that there is a strong unique continuation principle; specifically, if $\lambda \in L^\infty$ and $\mu \in C^{0,1}$, in $\Rn$ with $n\geq 2$, if a solution is zero in a neighborhood of any point, it is identically zero. For the elastic wave equation, however, the situation is somewhat more subtle; while some results exist, we will have to make assumptions about the Lam\'e parameters to acquire satisfactory results, the topic of the next section.
\section{Unique Continuation and Conditions on the Lam\'e Parameters}
One of the key attributes distinguishing the elastic wave equation from scalar wave equations is the presence of two speeds of propagation, the so-called P- and S-waves or modes (also known as compression and shear waves or modes).  In particular, these two speeds $c_1(x) = \sqrt{2\mu + \lambda}$  and $c_2(x) = \sqrt{\mu}$ correspond to the eigenvalues of the principal symbol of $P$:
\begin{align*}
	p(x,\xi) = (\lambda + 2\mu)(x) \xi \xi^T + \mu(x) \left(\abs{\xi}^2 I - \xi \xi^T \right).
\end{align*}
It will be useful to define scalar wave operators corresponding with these speeds, and so with $a_1 = 1 / (2\mu + \lambda)$ and $a_2 = 1 / \mu$, we define $\square_{a_j}, j = 1,2$ by
\begin{align*}
	\square_{a_j} := a_j \partial_t^2 - \Delta.
\end{align*}
The property of finite speed of propagation for the elasticity system is essentially the same as the scalar case, as demonstrated by the following definition and theorem.
\begin{definition}
We will say that $u$ has \textit{finite speed of propagation} in $ (0,T) \times B(x_0, \epsilon)$, with maximum speed $c >0$, if for any $t_0 \in [0,T)$, $u(t_0,\cdot) = u_t(t_0,\cdot) = 0$ in $B(x_0,\epsilon)$ implies $u = 0$ a.e. in the cone $\cup_{0<s<\epsilon/c}C_s$, where $C_s = \set{t = t_0 + s}\times B(x_0,\epsilon - cs)$.
\end{definition}

\begin{theorem} Assume that $\mu, \lambda \in C^2(\R^3)$, and suppose $u \in H^2$ solves \eqref{PDEu}.  Then, for any open ball $B(x_0,\epsilon) \subset \R^3$, $u$ has finite propagation speed in $(0,T) \times B(x_0,\epsilon)$, with maximum speed $c = \sup_{x \in B(x_0,\epsilon)} \sqrt{2\mu + \lambda}$.
\label{speed}
\end{theorem}
For a proof, we direct the reader to \cite{MY}.  With this result in mind, we will hereafter assume that there exist constants $c^-, c^+$  so that
\begin{align}
c^+ = \sup_{x \in \R^3} \sqrt{2\mu + \lambda}<\infty,\\
c^- = \inf_{x \in \R^3} \sqrt{\mu} > 0.
\end{align}
We will also assume that $(\Omega, P)$ is \textit{non-trapping}, meaning (recalling the definition and discussion for the scalar case given on page~\pageref{microlocal}) that every ray (i.e. the projection of every bicharacteristic starting in $\Omega$) leaves $\Omega$ in finite time, and that the supremum of these times, which we will again call $T(\Omega)$, is also finite. 

The presence of two speeds of propagation causes some difficulty regarding questions of unique continuation. For a homogeneous medium (i.e., the Lam\'e parameters are constant), the system is diagonalizable, and the P- and S-modes are preserved throughout the wave's evolution, effectively reducing this problem to the scalar case. In the more general setting however, P-waves may transmit or reflect as S-waves (and vice-versa) at an interface, making the question of unique continuation more subtle. For instance, if the P-wave vanishes on the boundary of a set for all time, one cannot necessarily conclude it vanishes on the interior as well; it may simply be transmitting as an S-wave instead. The extent to which the two modes can be decoupled is useful for understanding the reflection and transmission of singularities, and has been studied using the pseudodifferential calculus (see \cite{TAY2} and \cite{TAY} for smooth $\mu$ and $\lambda$, and most recently, \cite{BHSU} for $\mu,\lambda \in C^{1,1}$).

Returning to unique continuation, many satisfactory results have been proven via Carleman estimates, and a thorough discussion can be found in the work of Eller, Isakov, Nakamura and Tataru, as well as that of Cheng, Isakov, Yamamoto and Zhou (see \cite{EINT} and \cite{CIYZ}, respectively) and other authors.  In the following two theorems, $u$ will be a solution of \eqref{PDEu}, $\Omega '$ is an open domain in $\R^3$, and $T'$ will be some positive real number (in practice, we will have $\Omega ' $ containing $\Omega$, and $T'$ larger than $T(\Omega)$). Both of these results are from \cite{EINT}; the former is a slightly simplified statement of Corollary 3.5, and the latter is Theorem 5.5.
\begin{theorem}(Eller, Isakov, Nakamura, Tataru)\\
Let $a_1 = \frac{1}{\mu}$, $a_2 = \frac{1}{\lambda + 2\mu}$ and assume, for some $\theta > 0$, they both satisfy
\begin{align}
	\theta^2 a_j ( a_j + a_j^{-1/2} \abs{t\grad a_j} ) < a_j + 1/2 x \cdot \grad a_j
\label{condition1}
\end{align}
and
\begin{align}
	\theta^2 a_j \leq 1
\label{condition2}
\end{align}
on $[- T', T'] \times \overline{\Omega '}$, that $a_j \in C^1([-T', T'] \times \overline{\Omega '})$, and that $\Omega ' \subset B(0, \theta T')$.
Then, if $u=0$ and $\partial_\nu u = 0$ on $(-T', T') \times \partial \Omega '$, then $u(t,x)=0$ when $\abs{x}^2 > \theta^2 t^2$.
\label{tataru1}
\end{theorem}
\begin{theorem}
Assume that the coefficients $\mu, \lambda \in C^3$ are time independent.  Let $S$ be a noncharacteristic surface with respect to both  $\square_{a_1}$ and $\square_{a_2}$.  Then we have unique continuation across $S$ for $H^1_{loc}$ solutions $u$ to \eqref{PDEu}.
\label{tataru2}
\end{theorem}
With the assistance of these theorems, we conclude this section by proving a unique continuation result needed later for the reconstruction process. Essentially, we seek to answer the following question: 

Suppose a solution $u$ of \eqref{PDEu}, with $u(0,x) = f(x) = 0$ outside $\Omega$, also vanishes outside $\Omega$ at some later time $T$. From this, can we determine if $f(x) = 0$ inside $\Omega$ as well? Because of finite propagation speed, this is certainly not the case for all times $T>0$ (let $T$ be small, and let the support of $f$ be, for instance, some small ball contained in $\Omega$), and so we seek to describe sufficient conditions for making such a determination. 
\begin{theorem}[Sufficient conditions for the Lam\'e Parameters]
Suppose that $\mu,~\lambda~\in~C^3$ are time independent, and that the maximum and minimum speeds of propagation satisfy the inequality
\begin{align}
	c^+ < 3 c^-,
\end{align}
that there exist $\theta$, $T$, and $\epsilon > 0$ so that
\begin{align}
	\frac{1}{3} c^+ < \theta < c^-\\
	T > \frac{2(R+\epsilon)}{3\theta - c^+}
\end{align}
and that $(a_1,a_2,T,\theta)$ satisfy the gradient condition (\ref{condition1}) on
\begin{align*}
	 \left[ -\frac{3T}{2}, \frac{3T}{2} \right] \times \overline{B\left(0,R + \frac{T}{2}c^+ + \epsilon \right)}.
\end{align*}
Assume also that the surface $S = \set{(t,x): \abs{x}^2 = \theta^2 t^2}$ is non-characteristic for $\square_{a_1}, \square_{a_2}$.
Then, for solutions $u$ of \eqref{PDEu} (with $f$ compactly supported in $\Omega = B(0,R)$), if $u(T,x) = 0$ for $x \notin \Omega$, we have $f(x) = 0$.
\label{uniqueness}
\end{theorem}
\begin{proof}
By assumption, we know
\begin{align*}
	u(T,x) = 0, \quad \textrm{for } x \notin \Omega,
\end{align*}
and since $f$ is compactly supported in $\Omega$, we know
\begin{align*}
	u(0,x) = 0, \quad \textrm{for } x \notin \Omega
\end{align*}
as well.  Thus, by finite speed of propagation, we then have both
\begin{align*}
	u(t,x) = 0 \quad \textrm{when }\abs{x} - R > c^+\abs{T-t}
\end{align*}
and
\begin{align*}
	u(t,x) = 0 \quad \textrm{when }\abs{x} - R> c^+\abs{t}.
\end{align*}
Combining these observations shows (see figure \ref{fig1})
\begin{align*}
		u(t,x) = 0 \quad \textrm{when }\abs{x} - R > \frac{T}{2}c^+, \ -T/2 \leq t \leq 3T/2.
\end{align*}
Next, by time-reversal, $u$ extends to an even function of $t$, so, in fact, we have
\begin{align*}
		u(t,x) = 0 \quad \textrm{when }\abs{x} - R > \frac{T}{2}c^+, \ -3T/2 \leq t \leq 3T/2.
\end{align*}
Now, because $\theta < c^-$, we have $\theta^2 a_j \leq 1$ for $j=1,2$, and by hypothesis, the gradient condition is satisfied on $\overline{B\left(0,R + \frac{T}{2}c^+ + \epsilon\right)}\times \left[ -\frac{3T}{2}, \frac{3T}{2} \right]$.  Because $\theta > c^+/3$ and $T > 2(R + \epsilon)/(3\theta - c^+)$, we check that
\begin{align}
	R + \frac{T}{2}c^+ +\epsilon < \frac{T}{2}(3\theta - c^+) + \frac{T}{2}c^+ < \frac{3T}{2}\theta 
\label{inequality}
\end{align}
so that we can apply Theorem \ref{tataru1}, with $\Omega ' = B(0,R + \frac{T}{2}c^+ + \epsilon)$ and $T' =\frac{3T}{2}$.  Thus, $u(t,x) = 0$ whenever $\abs{x}^2 > \theta^2 t^2$, showing that $u(0,x) = 0$, except possibly at the origin.  Theorem \ref{tataru2} allows us to extend this solution uniquely, and thus $u(0,0) = 0$ as well, showing $f \equiv 0$.
\end{proof}
\begin{figure}
\begin{center}
\includegraphics[width=10cm]{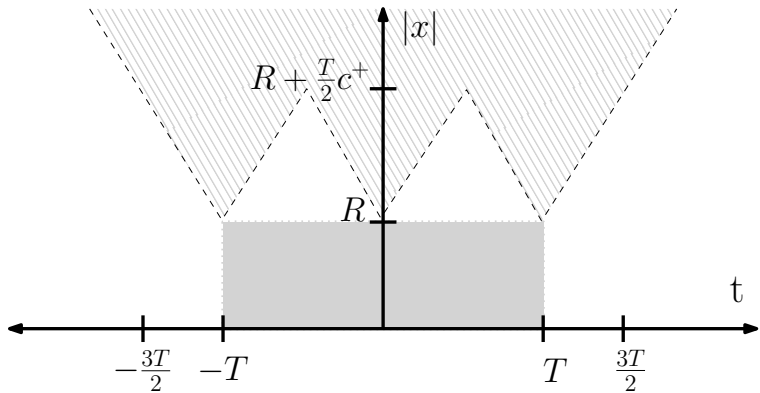}
\end{center}
\caption{The solid gray region is $B(0,R)\times(-T,T)$; the dashed region where $u=0$ by finite speed of propagation and time-reversal.}
\label{fig1}
\end{figure}
Before we continue, we make a few remarks about the conditions placed on the Lam\'e parameters in this theorem.  First, the condition $c^+ < 3c^-$ is needed to guarantee we can choose $\theta$ so that $\frac{1}{3} c^+ < \theta < c^-$; we need $\theta < c^-$ to ensure condition \eqref{condition1} of Theorem \ref{tataru1} is met, and $\frac{1}{3} c^+ < \theta$ is needed to ensure that the first inequality in \eqref{inequality} is valid.  This condition essentially reads, ``the speed of the S-wave must be more than a third that of the P-wave.''  For many materials, this is a realistic assumption (for instance, this is true of most earth materials), though as McLaughlin and Yoon note in \cite{MY}, in biological tissue, this assumption may not be reasonable. In spite of this, because the observation surface is not a true discontinuity in the medium, it may be reasonable to expect that there is no exchange between the P- and S-modes, allowing us to work with a stronger form of unique continuation and drop the assumptions on the wave speeds. We will not pursue this issue any further in the present work, however.

The gradient condition is somewhat more subtle, though making some additional hypotheses could help to simplify the situation.  For instance, if we were able to assume that the medium is homogeneous outside some neighborhood of $\Omega$ (i.e., $\lambda$ and $\mu$ are constant, so that $\grad a_j = 0$), then the gradient condition simply reads $\theta < c^-$ outside that neighborhood.  Alternately, if we were to make assumptions on the size of $\abs{\grad a_j}$, say $\abs{\grad a_j} < \delta a_j$ for some $\delta > 0$, then the condition could be reduced to
\begin{align*}
	\theta^2 (a_j + a_j^{1/2}\delta\abs{t}) + \frac{\delta}{2}\abs{x} < 1.
\end{align*}
Again, at this point in time, we will leave such additional assumptions for future work, and continue with the hypotheses in the form stated in the theorem.
\section{Energy of Initial Data and Solutions}
As we have previously discussed, the success of the reconstruction process depends on good local energy decay, ensured by the non-trapping condition. Before we move on to the reconstruction process, we will find it useful to specifically define two kinds of energy spaces associated with \eqref{PDEu}; first for the initial data, and then for solutions.  In what follows, $U$ is a domain in $\R^3$ (in practice, we will have $U = \Omega$ or $U = \R^3$).  We will begin by defining an inner product for the first space.
\begin{definition}
	For functions $f,g$, let
	\begin{align*}
	(f,g)_{H_D(U)} = \int_U \lambda(\divergence f)(\divergence g) + \mu \ \tr \left((\grad f)(\grad g) + (\grad f)^T(\grad g) \right) \ dx,
	\end{align*}
where $\tr$ indicates the trace.  Specifically, we have
	\begin{align*}
		\tr (\grad f)(\grad g) = \sum_{i,j = 1}^3 \frac{\partial f_j}{\partial x_i}\frac{\partial g_i}{\partial x_j}; \quad \tr (\grad f)^T(\grad g) = \sum_{i,j = 1}^3 \frac{\partial f_i}{\partial x_j}\frac{\partial g_i}{\partial x_j}.
	\end{align*}
\end{definition}
It is easy to see that $(\cdot,\cdot)_{H_D(U)}$ is symmetric and bilinear.  Furthermore,
\begin{lemma}
	For functions $f,g \in C^\infty_{0}(U)$, we have
	\begin{align*}
	(f,g)_{H_D(U)} = \innp{Pf}{g}_{L^2(U)} = \innp{f}{Pg}_{L^2(U)}.
	\end{align*}
\end{lemma}
\begin{proof}
Using the identities
\begin{align*}
	\left( \divergence (\mu (\grad f))\right)\cdot g = \divergence ((\mu (\grad f))g) - \mu \ \tr (\grad f)(\grad g),
\end{align*}
\begin{align*}
	\left( \divergence (\mu (\grad f)^T)\right)\cdot g = \divergence ((\mu (\grad f)^T)g) - \mu \ \tr (\grad f)^T(\grad g),
\end{align*}
and
\begin{align*}
	\left(\grad(\lambda (\divergence f))\right) \cdot  g = \divergence (\lambda (\divergence f) g) - \lambda(\divergence f)(\divergence g), 
\end{align*}
we have
\begin{align*}
	\innp{Pf}{g}_{L^2(U)} &= \int_U \lambda(\divergence f)(\divergence g) + \mu \ \tr \left((\grad f)(\grad g) + (\grad f)^T(\grad g) \right) \ dx \\
& \quad - \int_U \divergence ((\mu((\grad f) + (\grad f)^T)g)) + \divergence (\lambda (\divergence f) g) \ dx
\end{align*}
Since $g \in C^\infty_0(U)$, the second integral is zero by the divergence theorem, and we have $\innp{Pf}{g}_{L^2(U)} = (f,g)_{H_D(U)}$.  Similarly, using $f \in C^\infty_0(U)$ will show $(f,g)_{H_D(U)} = \innp{f}{Pg}_{L^2(U)}$.
\end{proof}

Next, we define $H_D(U)$ to be the completion of $C_0^\infty(U)$ under the norm
\begin{align*}
	\norm{f}_{H_D(U)}^2 = (f,f)_{H_D(U)}
\end{align*}
(this quantity is essentially the total elastic energy of $f$ in $U$, with a factor of $1/2$ omitted for convenience), and then define the energy space
\begin{align*}
	\mathcal{H}(U) = H_D(U)\oplus L^2(U)
\end{align*}
with the norm
\begin{align*}
	\norm{(f_1,f_2)}_{\mathcal{H}(U)} = \norm{f_1}_{H_D(U)} + \norm{f_2}_{L^2(U)}.
\end{align*}
Finally, we define the total energy for a function $u(t,x)$ at time $t$ as
\begin{align*}
	E_{U}(u,t) = \norm{u(t,\cdot)}_{H_D(U)}^2 + \norm{u_t(t,\cdot)}_{L^2(U)}^2,
\end{align*}
i.e., the total energy is the sum of the elastic and kinetic energies of $u$ at time $t$.  Before moving on, we note that, as a consequence of these definitions, if $u$ solves \eqref{PDEu} with $f \in \HDO$, then we have
\begin{align*}
	E_{\R^3}(u,0) =  \norm{f}_{H_D(\R^3)}^2 = \norm{f}_{\HDO}^2 = E_{\Omega}(u,0).
\end{align*}
\section{Reconstruction}
Now that we have all the necessary preliminaries in place, we are in the position to investigate the reconstruction process.  As suggested by our earlier discussion of techniques in TAT, one possible approach to reconstruction would be the use of time-reversal for solutions $v_0$ of \begin{align}
	\begin{cases}
		(\partial_t^2 + P) v_0 &= 0 \quad \textrm{  in  } (0,T) \times \Omega, \\
		v_0|_{[0,T]\times \partial \Omega} &= h, \\
		v_0|_{t=T} &= 0, \\
		\partial_t v_0 |_{t=T} &= 0,
	\end{cases}
\end{align}
where $h = \Lambda f$, and then attempt to define an inverse $A_0$ by
\begin{align*}
	A_0h := v_0(0,\cdot) \quad \textrm{in } \bar \Omega,
\end{align*}
with the hope that $A_0 h = A_0\Lambda f$ approximates $f$.  The problem with this method is that $h$ may not vanish on $\set{t = T} \times \partial \Omega$, causing the boundary conditions to be incompatible.

To correct this, we will need to introduce an error term, and modify our approach accordingly.  Given $h$ (again, eventually $h = \Lambda f$), define $\phi$ and $v$ by
\begin{align*}
	P\phi = 0, \quad \phi|_{\partial \Omega} = h(T,\cdot),
\end{align*}
\begin{align}
	\begin{cases}
		(\partial_t^2 + P) v &= 0 \quad \textrm{  in  } (0,T) \times \Omega, \\
		v|_{[0,T]\times \partial \Omega} &= h, \\
		v|_{t=T} &= \phi, \\
		\partial_t v |_{t=T} &= 0,
	\end{cases}
\label{PDEv}
\end{align}
and note that the boundary data is now compatible to first order, in that $\phi = h$ on $\set{t = T} \times \partial \Omega$.  Now, we define the pseudo-inverse $A$ by
\begin{align*}
	Ah := v(0,\cdot) \quad \textrm{in } \bar \Omega.
\end{align*}
Our goal now is to show that $A$ maps the range of $\Lambda$ to $\HDO$, and that the error $\norm{f - A\Lambda f}_{\HDO}$ is small compared to $\norm{f}_{\HDO}$; in particular, we will show that the operator $K= (I - A\Lambda)$ is a contraction on $\HDO$.  The following theorem is our main result. Recall that we say $(\Omega,P)$ is \textit{non-trapping} if every ray starting in $\Omega$ exits in finite time, and the supremum of these times (which we denote by $T(\Omega)$) is finite as well (see page~\pageref{microlocal} for more discussion).
\begin{theorem} Suppose that $f$ is compactly supported in $\Omega$, that $(\Omega,P)$ is non-trapping and satisfies the hypotheses of Theorem \ref{uniqueness}, with $T > T(\Omega)$.  Then $A\Lambda = I- K$, where $K$ is compact in $\HDO$ and $\norm{K}_{\HDO} < 1$.  In particular, $I - K$ is invertible on $\HDO$,  and we have the following Neumann expansion for $f$:
\begin{align*}
	f = \sum_{m = 0}^\infty K^mAh, \quad h := \Lambda f.
\end{align*}
\end{theorem}
\begin{proof}
Let $w$ solve
\begin{align}
	\begin{cases}
		(\partial_t^2 + P) w &= 0 \quad \textrm{  in  } (0,T) \times \Omega, \\
		w|_{[0,T]\times \partial \Omega} &= 0, \\
		w|_{t=T} &= u|_{t=T} - \phi, \\
		\partial_t w |_{t=T} &= \partial_t u |_{t=T}.
	\end{cases}
\end{align}

Let $v$ be the solution of \eqref{PDEv} with $h = \Lambda f$.  Then $v + w$ will solve same initial boundary value problem as $u$ (with initial conditions given at $t = T$), and thus $u = v + w$.  For $t = 0$, we then have
\begin{align*}
	f = A \Lambda f + w(0,\cdot),
\end{align*}
and so,
\begin{align*}
	Kf = w(0,\cdot).
\end{align*}

For convenience of notation, let $u^T = u(T,\cdot)$ and $u_t^T = u_t(T,\cdot)$. Because $u^T = \phi$ on $\partial \Omega$, and because $P\phi = 0$,
\begin{align*}
	(u^T - \phi,\phi)_{\HDO} = \innp{u^T - \phi}{P\phi}_{\LTO} = 0.
\end{align*}
Thus,
\begin{align*}
	\norm{u^T -\phi}_{\HDO}^2 = \norm{u^T}_{\HDO}^2 - \norm{\phi}_{\HDO}^2 \leq \norm{u^T}_{\HDO}^2,
\end{align*}
showing that
\begin{align*}
	E_\Omega(w,T) = \norm{u^T - \phi}_{\HDO}^2 + \norm{u_t^T}_{\LTO}^2 \leq \norm{u^T}_{\HDO}^2 + \norm{u_t^T}_{\LTO}^2 = E_{\Omega}(u,T).
\end{align*}

Because $w = 0$ on $\partial \Omega$, we have $(w,w)_{H_D(\Omega)} = \innp{Pw}{w}_{\LTO}$ so that
\begin{align*}
	\frac{d}{dt} E_{\Omega}(w,t) &= \frac{d}{dt}\innp{Pw}{w}_{\LTO} + \frac{d}{dt}\innp{w'}{w'}_{\LTO}\\
	&= \innp{Pw'}{w}_{\LTO} + \innp{Pw}{w'}_{\LTO} + \innp{w''}{w'}_{\LTO} + \innp{w'}{w''}_{\LTO} \\
	&= \innp{Pw'}{w}_{\LTO} + \innp{w'}{w''}_{\LTO} \\
	&= \innp{w'}{Pw}_{\LTO} + \innp{w'}{w''}_{\LTO} = 0
\end{align*}
(note that this is just a reflection of the fact that the Dirichlet boundary conditions imposed on $w$ are energy-conserving).  Therefore,
\begin{align*}
	E_\Omega(w,0) = E_\Omega(w,T) \leq E_\Omega(u,T) \leq E_{\R^3}(u,T) = E_{\Omega}(u,0) = \norm{f}_{\HDO}^2,
\end{align*}
showing
\begin{align*}
	\norm{Kf}_{\HDO}^2 = \norm{w(0,\cdot)}_{\HDO}^2 \leq E_\Omega(w,0) \leq \norm{f}_{\HDO}^2.
\end{align*}

We next seek to show that this inequality is, in fact, strict.  So, suppose that there is some $f$ so that $\norm{Kf}_{\HDO} = \norm{f}_{\HDO}$.  This implies that all the above inequalities are actually equalities as well, and in particular, $E_\Omega (u,T) = E_{\R^3} (u,T)$, so that
\begin{align*}
	u(T,x) = 0, \quad \textrm{for } x \notin \Omega.
\end{align*}
By Theorem \ref{uniqueness}, we then have $f = 0$,  and thus $\norm{Kf}_{\HDO} < \norm{f}_{\HDO}$ for all non-zero $f \in \HDO$.

Now, we will show that $K$ is compact.  We have that $u(T,\cdot)$ and $u_t(T,\cdot)$ are smooth as functions on $\bar \Omega$ because $T > T(\Omega)$ (all singularities beginning in $\Omega$ will have left), and so, as linear operators acting on $f$, they have smooth Schwarz kernels.  This will also imply $\phi$ is smooth, by elliptic regularity.  From this, we can see that the map
\begin{align*}
	\HDO \to \mathcal H(U) : f \mapsto (u^T - \phi, u^T_t)
\end{align*}
has a smooth kernel, and is therefore compact.  Next, we note that the solution operator of \eqref{PDEv} from $t=T$ to $t=0$ (i.e., the map $ (u^T-\phi,u^T_t)\mapsto(w(0,\cdot),w_t(0,\cdot)$) is bounded in $\mathcal H (U)$ (unitary, actually, because $E_{\Omega}(w,T) = E_{\Omega}(w,0)$).  Thus, we  conclude that the map $K: \HDO \to \HDO, f \mapsto w(0,\cdot)$ is compact, as the composition of a compact and a bounded map, finally allowing us to conclude $\norm{K}_{\HDO} <1$, and the proof is complete.
\end{proof}
It is additionally worth noting that, in addition to uniqueness, the proof of this theorem also gives insight into the stability of this problem. In particular, we have
\begin{align*}
	\norm{Kf}_{\HDO} \leq \left( \frac{E_\Omega(u,T)}{E_\Omega(u,0)} \right)^{1/2}\norm{f}_{\HDO}, \quad \forall f \in \HDO, f\neq 0,
\end{align*}
providing a bound on $\norm{K}_{\HDO}$ in terms of the local energy decay. Furthermore, because $f - A\Lambda f = Kf$, this also describes the error in the reconstruction if we only use the first term $K^0Ah = A\Lambda f$ of the Neumann series. Additionally, if $T$ is chosen so that the hypotheses of Theorem \ref{uniqueness} are satisfied, but $ T \leq T(\Omega)$, then we can still conclude that $\norm{Kf}_{\HDO}~<~\norm{f}_{\HDO}$ for all non-zero $f$, but we can no longer be sure $K$ is compact, and thus have no reason to expect $\norm{K}_{\HDO}$ is \textit{strictly} less than $1$. Finally, if $T < \frac{R}{c^+}$, then one can find a non-zero $f$ so that $\norm{Kf}_{\HDO} = \norm{f}_{\HDO}$ (by choosing an $f$ supported in a small enough ball, finite speed of propagation will imply $u(T,x) = 0$ outside of $\Omega$).
\section{Conclusion}
While we have suggested a method of reconstruction for thermoacoustic tomography in elastic media, and demonstrated the uniqueness and stability of this method, there is still much work to be done.

Indeed, we have only explored the case of complete data, and in practical applications, such as breast imaging and geophysical imaging, one can only expect to be able to measure data on a portion of the boundary. Thus, examining the problem of incomplete data will be critical to the practical relevance of the elastic approach, and should be explored in future work. As previously mentioned, various authors (again, see \cite{SU}, \cite{XKA}, \cite{XWKA}) have investigated this problem for the acoustic wave equation, and produced satisfactory results; perhaps by adapting these methods, good results can be obtained for setting of elastic media as well.

Furthermore, we have worked with the restriction that $c^+ / 3 < c^-$, which, as mentioned, is a reasonable hypothesis for earth materials, but not necessarily for biological tissue. This assumption may be stricter than actually needed, and so it would be nice to find an alternative to using \cite{EINT} in the uniqueness step. At this point, it is not clear how to accomplish this, but perhaps by adapting the methods (rather than the results) of Eller, Isakov, Nakamura and Tataru some progress can be made. Indeed, in their work the coefficients $a_j$ are generally allowed to vary in time and be of relatively limited smoothness (e.g., $C^1$); possibly by relaxing these assumptions, a better unique continuation result can be obtained.

Finally, developing computer simulations and attaining numerical data based on the method proposed in this paper are an obvious next step for this approach to mature from theory to application. As we discussed earlier, this has already been done for the acoustic case (in \cite{QSUZ} and others), and hopefully computational methods can be developed for elastic media in the near future.
\section*{Acknowledgments}
This research was generously supported under NSF grant DMS-0838212.

\end{document}